\documentclass[12pt]{amsart}

\usepackage{amsmath}
\usepackage{amsthm}
\usepackage{stmaryrd}
\usepackage{amsfonts,mathrsfs,amssymb}%,txfonts}
\usepackage{times}
\usepackage{fullpage}
\title{Bargmann-Fock Extension From Singular Hypersurfaces}
\author{Vamsi Pritham Pingali}
\email{vpingali@math.jhu.edu}
\author{Dror Varolin$^{\dagger}$} 
\address{Krieger 412, Dept of Mathematics, Johns Hopkins University, Baltimore, MD 21210 USA}
\thanks{$\dagger$ Partially supported by the NSF}
\email{dror@math.sunysb.edu}
\address{Dept of Mathematics, Stony Brook University, Stony Brook, NY 11794-3651 USA}

\newcommand{\noi}{\noindent}

\newcommand{\ch}{{\mathcal H}}

\newcommand{\co}{{\mathcal O}}

\newcommand{\cR}{{\mathcal R}}

\newcommand{\sC}{{\mathscr C}}

\newcommand{\se}{{\mathscr E}}

\newcommand{\sh}{{\mathscr H}}
\newcommand{\si}{{\mathscr I}}

\newcommand{\sr}{{\mathscr R}}

\newcommand{\fH}{{\mathfrak H}}

\newcommand{\vp}{\varphi} 
\newcommand{\ve}{\varepsilon}

\newcommand{\B}{{\mathbb B}}
\newcommand{\C}{{\mathbb C}}

\newcommand{\R}{{\mathbb R}}

\newcommand{\Z}{{\mathbb Z}}

\newcommand{\vol}{\operatorname{Vol}}
\newcommand{\dist}{\operatorname{dist}}

\newcommand{\red}{\hfill $\diamond$}

\newcommand{\di}{\partial}
\newcommand{\dbar}{\bar \partial}

\newcommand{\re}{{\rm Re\ }}

\newcommand{\ii}{\sqrt{-1}}

\newcommand{\emb}{\hookrightarrow}

\def\Xint#1{\mathchoice 
{\XXint\displaystyle\textstyle{#1}}
{\XXint\textstyle\scriptstyle{#1}} 
{\XXint\scriptstyle\scriptscriptstyle{#1}}
{\XXint\scriptscriptstyle\scriptscriptstyle{#1}}
\!\int} 
\def\XXint#1#2#3{{\setbox0=\hbox{$#1{#2#3}{\int}$} 
\vcenter{\hbox{$#2#3$}}\kern-.5\wd0}} 
 
\def\dashint{\Xint-} 

\usepackage{hyperref}

\begin{document}
%\setlength{\textwidth}{24cm}
%\setlength{\textheight}{30cm}
%\maketitle

\theoremstyle{plain}
\newtheorem{thm}{\sc Theorem}
\newtheorem*{s-thm}{\sc Theorem}
\newtheorem{lem}{\sc Lemma}[section]
\newtheorem{d-thm}[lem]{\sc Theorem}
\newtheorem{prop}[lem]{\sc Proposition}
\newtheorem{cor}[lem]{\sc Corollary}

\theoremstyle{definition}
\newtheorem{conj}[lem]{\sc Conjecture}
\newtheorem{prob}[lem]{\sc Open Problem}
\newtheorem{defn}[lem]{\sc Definition}
\newtheorem{qn}[lem]{\sc Question}
\newtheorem{ex}[lem]{\sc Example}
\newtheorem{rmk}[lem]{\sc Remark}
\newtheorem{rmks}[lem]{\sc Remarks}
\newtheorem*{ack}{\sc Acknowledgment}

%\begin{center}
%{\Large {\bf Restricting Weighted-Square-Integrable Holomorphic Functions to \\ 
%\vskip .08in

%Hypersurfaces in Complex Euclidean Space}}
%\\ \ 
%\\ \ 
%{\large {\sf Stanislav Ostrovsky \quad {\small {\rm and}} \quad Dror Varolin\footnote{Partially supported by an NSF grant}}}

%\end{center}

\begin{abstract}
We establish sufficient conditions for extension of weighted-$L^2$ holomorphic functions from a possibly singular hypersurface $W$ to the ambient space $\C ^n$.  The $L^2$-norms we use are the so-called generalized Bargmann-Fock norms, and thus there are restrictions on the singularities of $W$ as well as the density of $W$.  Our sufficient conditions are that $W$ has density less than $1$ and is uniformly flat in a sense that extends to singular varieties the notion of uniform flatness introduced in \cite{osv}.  We present an example of Ohsawa showing that uniform flatness is not necessary for extension in the singular case,  and find an example showing that, for rather different reasons, uniform flatness is also not necessary in the smooth case.  The latter answers in the negative a question posed in \cite{osv}.  
\end{abstract}

\maketitle

\setcounter{tocdepth}1

%\tableofcontents

%\setcounter{section}{-1}

\section*{Introduction}

In this article we consider the problem of extending, from an analytic hypersurface $W$ in $\C ^n$, holomorphic functions that are square-integrable with respect to some ambient weight, in such a way that the extension is also square-integrable with respect to the same weight.  When the hypersurface is smooth and uniformly flat, the result we present here was proved in \cite{osv}:  If $W$ is uniformly flat and the density of $W$ is less than $1$ then extension is possible.  (See Sections \ref{flat-section} and \ref{density-section} for the definition of uniform flatness  and of density respectively.)  On the other hand, if $W$ is singular then in general extension is not possible.  The precise--- by which we mean necessary and sufficient--- conditions for such extension on a possibly singular hypersurface are not even conjectured.  

Below we extend the notion of uniformly flat hypersurface to a possibly singular hypersurface.  The notion places strong restrictions on the singularities of the hypersurface, among other things.  We show that the results of \cite{osv} extend to the case of uniformly flat singular varieties. 

To state our results precisely, we introduce some notation.  Let $\omega := \tfrac{\ii}{2} \di \dbar |z|^2$ denote the K\"ahler form associated to the Euclidean metric.  To a smooth function $\vp : \C ^n \to \R$ we associate the Hilbert space 
\[
\sh (\C ^n,\vp) := \co (\C ^n ) \cap L^2 (e^{-\vp}) = \left \{ f \in \co (\C ^n) \ ;\ \int _{\C ^n} |f|^2 e^{-\vp} \omega ^n < +\infty\right \}.
\]
Let $W$ be a possibly singular complex analytic hypersurface.  To $\vp$ and $W$ we associate the Hilbert space 
\[
\fH (W,\vp) := \left \{ f \in \co (W)\ ;\ \int _{W_{\rm reg}} |f|^2 e^{-\vp} \omega ^{n-1} < +\infty \right \}.
\]
The different letters $\sh$ and $\fH$ stress that in the latter case, the weight $\vp$ is defined on the entire ambient space $\C ^n$ that contains $W$.  We emphasize that, by definition, a function is holomorphic on $W$ if each point $x \in W$ has a neighborhood $U$ in $\C ^n$ and a holomorphic function $\tilde f \in \co (U)$ such that $\tilde f |_W = f$.  Since $W \subset \C ^n$ is a closed analytic subset, we may even take $U= \C ^n$.

Let us denote by $\sr _W : \sh (\C ^n,\vp) \to \fH(W,\vp)$ the map that sends $F\in \sh (\C ^n,\vp)$ to its restriction to $W$.  In general, this restriction map is not bounded.  For example, when $n=1$, $\sr _W$ is bounded if and only if $W$ is a finite union of uniformly separated sequences.  But we will not discuss the boundedness of $\sr _W$ in this article; our main concern is with the surjectivity of $\sr _W$.

In section \ref{flat-section} we define the notion of uniformly flat complex analytic hypersurface $W$ and in Section \ref{density-section} the upper density $D^+_{\vp}(W)$ of a hypersurface with respect to a weight $\vp$.  The notions of uniform flatness and upper density were defined for smooth hypersurfaces in \cite{osv}.   Here we introduce modifications of both definitions to the case of possibly singular varieties.

We can now state our first main result.

\begin{thm}\label{suff-analytic}
Let $\vp:\C ^n \to \R$ be a $\sC ^2$-smooth function satisfying 
\begin{equation}\label{gbf-weight}
\ve \omega \le \ii \di \dbar \vp \le C\omega
\end{equation}
for some positive constants $\ve$ and $C$, and let $W \subset \C ^n$ be a possibly singular, uniformly flat complex hypersurface such that $D^+_{\vp}(W) < 1$.  Then $\sr _W : \sh (\C ^n,\vp) \to \fH (W,\vp)$ is surjective.
\end{thm} 

In fact we will show that if $W$ is {\it smooth} and uniformly flat then condition \eqref{gbf-weight} can be dropped completely.  We therefore conjecture that the same is true for Theorem \ref{suff-analytic}.

The method used to prove Theorem \ref{suff-analytic} has two parts, the second of which bears similarity to the work in \cite{osv}.  The first part, which concerns local extension with estimates near the singularities, therefore constitutes one of the main contributions of this paper.  

The second main contribution of the present paper is to show that the conditions of Theorem \ref{suff-analytic} are not necessary; especially, uniform flatness need not hold in general.  We show this by example.  In the case of singular $W$, we present an example told to us by Ohsawa.  In the smooth case, we find a new example that took us rather by surprise when we first discovered it.  The example is the content of Theorem \ref{non-necess-smooth} below.

The examples we find suggest that the kinds of separation conditions we expect to constrain extension are ``in the large", rather than local.  The exact condition, which must reduce to the necessary condition of uniformly separated sequence in the case where $W$ is zero-dimensional, remains undiscovered so far as the authors know.

\begin{ack}
Thanks to Bo Berndtsson, Laszlo Lempert, Jeff McNeal, Takeo Ohsawa, Quim Ortega, Stas Ostrovsky and Andrew Young for many stimulating discussions.  The second author about the first example in Section \ref{non-necess} from Takeo Ohsawa at Oberwolfach in April 2009.  He is grateful to Professor Ohsawa for telling him this example, and to the MFO for providing a stimulating environment. 
\end{ack}

\section{Weighted mean-value inequalities}

In what follows, we will repeatedly use the following result.

\begin{lem}\label{quimbo-trick}\cite{quimbo,lindholm}
Let $\vp$ be a plurisubharmonic function on the unit ball $\B_k$ in $\C ^k$ such that $\ii \di \dbar \vp \le M\omega$ for some $M>0$.  Then there exist a positive constant $K$, depending only on M, and a holomorphic function on $G \in \co (\tfrac{1}{2}\B_k )$ such that $G(0)=0$ and
\[
\sup _{\B _k \left (0,1/2\right )} \left |\vp - \vp (0) - 2\re \ G \right | \le K.
\]
Moreover, if $\vp$ depends smoothly on a parameter, then so does $G$.
\end{lem}

Lemma \ref{quimbo-trick} shows that for each $z \in \C ^n$ and $r >0$ there is a function $G_z \in \co (B(z,r))$ such that 
\[
G_z(w) = \vp (w) - \vp (z) + \psi _z(w)
\]
where $\psi_z(w)$ is bounded on $B(z,r)$ and $G_z(z) = \psi _z(z) = 0$.  By averaging over $B(z,r)$ we see that, with 
\[
\vp _r (z) = \dashint _{B(z,r)} \vp (\zeta)\omega ^n (\zeta),
\]
we have the estimate
\[
\left | \vp (z) - \vp _r (z) \right | \le C_r.
\]
It follows that $\sh (\C ^n,\vp) = \sh (\C ^n, \vp _r)$ and $\fH(W,\vp) = \fH(W,\vp_r)$ in the sense that the identity map is a bounded vector space isomorphism.  We will use this uniform comparison between $\vp$ and $\vp_r$ constantly and often without mention.

We will also need some uniform and $\sC^1$ estimate for weighted-$L^2$ holomorphic functions in $\C ^n$.  These estimates, the first of which often also goes by the name {\it weighted Bergman inequality}, might be thought of as weighted analogues of the Cauchy estimates for a function and its derivative.

\begin{lem}\label{quimbo-cauchy-est}\cite{quimseep, lindholm}
Let $\psi$ be a function satisfying 
\[
-C\omega \le \ii \di \dbar \psi \le C\omega.
\]
Then there is a constant $K > 0$ such that for any $F \in \sh (\C ^n, \psi)$, 
\begin{eqnarray}
\label{norm} \sup _{\C ^n} |F|^2e^{-\psi} &\le& K\int _{\C ^n} |F|^2 e^{-\psi} \omega ^n\\
\nonumber \text{and} && \\
\label{d-of-norm} \sup _{\C ^n} |d(|F|e^{-\tfrac{1}{2} \psi})^r| &\le& K \left ( \int _{\C ^n} |F|^2 e^{-\psi} \omega ^n \right )^{r/2}.
\end{eqnarray}
\end{lem}

\noi Strictly speaking, (2) was probably not explicitly proved in the literature when $n \ge 2$, but a slight modification of the proof of Lemma \ref{quimbo-trick} in \cite{lindholm} can be used to obtain $\sC ^1$-estimates for the function $\psi _z(w)$ above, and thus generalize the proof of \cite{quimseep} to higher dimensions.  For details, the reader can see, for example, \cite[Lemma 2.1]{sv-toeplitz}

\section{Uniform flatness}\label{flat-section}

We extend to possibly singular hypersurfaces the notion of uniform flatness introduced for smooth hypersurfaces in \cite{osv}.

\begin{defn}
\begin{enumerate}
\item[(i)] For a subset $A \subset \C ^n$ and a positive number $\ve$, we define
\[
U_{\ve} (A) := \{ x\in \C ^n\ ;\ {\rm dist}(x,A) = \inf _{a\in A} |x-a| < \ve \}
\]
\item[(ii)] Let $Y \subset \C^n$ be a smooth complex hypersurface with boundary.  If $\ve : Y \to (0,\infty)$ is a continuous function, the union
\[
N_{\ve}(Y):= \bigcup _{y\in Y}\left \{y+t \tfrac{df(y)}{|df(y)|}\ ;\ t\in \C \text{ and } |t|< \ve (y), (f)=\co _{Y,y}\right \}
\]
is said to be a tubular neighborhood of $Y$ if it is diffeomorphic to a neighborhood of the zero section in the normal bundle of $Y$.
\red
\end{enumerate}
\end{defn}

\noi In the rest of the paper, the function $\ve$ in the definition of $N_{\ve}(Y)$ will always be constant.

Our goal in this section is to extend to certain singular varieties the notion of uniform flatness introduced in \cite{osv} for smooth hypersurfaces in $\C ^n$.  To motivate our definition, let us recall the notion of uniformly flat smooth hypersurfaces.

\begin{defn}[Uniform flatness.  Smooth case\cite{osv}]
A smooth hypersurface $W \subset \C ^n$ is said to be uniformly flat if there exists a positive constant $\ve _o$ such that $U_{\ve_o}(W) = N_{\ve _o}(W)$.
\red
\end{defn}

We take this opportunity to remind the reader of the following proposition describing the basic properties of uniformly flat smooth hypersurfaces.  We remind the reader of the notation 
\[
D_W(w,\ve_o) := T_{W,w} \cap B(w,\ve _o) \oplus \{v\in T_{\C ^n,w}\ ; \di f(w) v = 1 \text{ and } |v| <\ve _o \}.
\]

\begin{prop}\cite[Proposition 3.2]{osv}\label{osv-uf}
Let $W \subset \C ^n$ be a uniformly flat hypersurface and let $\ve _o$ be a constant such that $U_{\ve_o}(W) = N_{\ve _o}(W)$.  Then the following hold.
\begin{enumerate}
\item[(G)] Assume $n \ge 2$.  Then for all $w \in W$, $W \cap D_W(w,\ve_o)$ is given as a graph $y = f(x)$, over $T_{W,w}\cap B(w,\ve _o)$, of a function $f : T_{W,w} \cap B(w,\ve_o) \to \C$ satisfying 
\begin{equation}\label{quad-graph}
|f(w+x)| \le \frac{|x|^2}{\ve_o}.
\end{equation}
Here $D_W(w,\ve_o)$ denotes the union of the disks with centers on $T_{W,w}\cap B(w,\ve _o)$ and radius $\ve _o$ that are orthogonal to $T_{W,w} \subset \C^n$.
\item[(A)]  For each $R > 0$ there exists a constant $C_R>0$ such that for all $z \in \C ^n$
\[
{\rm Area}(W \cap B(z,R)) \le C_R.
\]
\end{enumerate}
\end{prop}

\noi In extending the notion of uniform flatness to the singular setting, we aim to achieve the following:
\begin{enumerate}
\item[(i)]  Away from the singular locus of $W$, the notion of uniform flatness should be the same as for general smooth varieties in $\C ^n$, and 
\item[(ii)] near the singular locus, the hypersurface should look like a finite number of uniformly flat smooth hypersurfaces intersecting pairwise-transversely, and the transversality should be uniform.  The local notion of uniform flatness will be the graph property (G). 
\end{enumerate}
With these comments in mind, we propose the following definition.

\begin{defn}[Uniform flatness]\label{u-flat-defn}
A singular analytic variety $W$ is said to be uniformly flat if there are numbers $\ve _o>0$ and $a>1$ with the following properties.
\begin{enumerate}
\item[(R)] The set $N_{\ve_o} (W- U_{a \ve_o}(W_{\rm{sing}}))$ is a tubular neighborhood of the smooth hypersurface with boundary $W- U_{a \ve_o}(W_{\rm sing})$ in $\C ^n$.
\item[(S)] For each $p \in W_{\rm sing}$ the set $W \cap B(p,\ve _o)$ is a union of smooth hypersurfaces $W_1,...,W_{N_p}$ each of which is given as the graph of a function on its tangent space with the property \eqref{quad-graph}.  Moreover,  
\begin{enumerate}
\item[(S$_N$)] $N_p \le \ve _o^{-1}$ for all $p$, and
\item[(S$_A$)] the angle at $p$ between any two of the $W_i$ lies in $[\ve _o, \pi - \ve _o]$.
\red
\end{enumerate}
\end{enumerate}
\end{defn}

\begin{rmk}
Note that $N_p \le n$ when $W$ has only simply normal crossing singularities.
\red
\end{rmk}

\begin{rmk}
At one point, we had hoped to use the following definition of uniform flatness:  a singular analytic variety $W$ was to be uniformly flat if there are positive numbers $\ve _o$ and $a$ such that the following holds.  For any $0<\ve <\ve _o$, the set 
\[
N_{\ve} (W- U_{a \ve}(W_{\rm{sing}})) 
\]
is a tubular neighborhood of the smooth hypersurface with boundary $W- U_{a \ve}(W_{\rm sing})$ in $\C ^n$.

So far, we have been unable to prove that this notion of uniform flatness is the same as the one we have taken above.  The more natural, but possibly less general, definition is the one made in this remark.  It would be nice to decide whether the two definitions are the same.  
\red
\end{rmk}

\noi The following are two important consequences of uniform flatness that we will use in the sequel.  These properties follow easily from Proposition \ref{osv-uf} and the definition of uniform flatness.

\begin{lem}\label{flat-properties} 
Let $W \subset \C ^n$ be a uniformly flat hypersurface, and let $\ve _o$ and $a$  be as in the definition of uniform flatness.  Then the following hold.
\begin{enumerate}
\item[(G)]  Assume $n \ge 2$.  Let $B_{W,w}(\ve _o):= T_{W,w} \cap B(w,\ve _o)$.  Then for all $w \in W-U_{(a+1)\ve_o}(W_{\rm sing})$, $W\cap N_{\ve_o}(B_{W,w}(\ve _o))$ is given as a graph over $B_{W,w}(\ve _o)$ of a function $f : B_{W,w}(\ve_o) \to \C$ satisfying
\[
|f(w+x)| \le \frac{|x|^2}{\ve _o}.
\]

\item[(A)] For each $R > 0$ there is a constant $C_R
>0$ such that for all $z \in \C ^n$,
\[
{\rm Area}(W \cap B(z,R)) \le C_R.
\]
\end{enumerate}
\end{lem}

\section{Density}\label{density-section}

\begin{defn}
Let $T \in \co (\C ^n)$ be a holomorphic function such that $W = T^{-1} (0)$ and $dT$ is nowhere zero on $W_{\rm{reg}}$.  (That is to say, $T$ generates the ideal of functions vanishing on $W$.)  For any $z\in\C^n$ and any $r>0$ consider the (1,1)-form
\[
\Upsilon ^W _r(z) := \frac{1}{2\pi} \sum _{i ,\bar j =1} ^n \left (\dashint _{B(z,r)} \frac{\di ^2\log |T|^2}{\di \zeta ^i \di \bar \zeta ^j} \omega ^n(\zeta) \right ) \ii dz^i \wedge d\bar z ^j.
\]
The $(1,1)$-form $\Upsilon _r^W(z)$ is called the total density tensor of $W$.
\red
\end{defn}

\begin{rmk}
Note that the total density tensor is identical in form to the total density tensor for smooth uniformly flat hypersurfaces introduced in \cite{osv}; the only thing different is that we no longer assume that $W$ is smooth.  As was pointed out then, the definition of $\Upsilon ^W_r (z)$ is independent of the choice of the function $T$ defining $W$.  Moreover, if $[W]$ denotes the current of integration over $W$ then $\Upsilon _r ^W$ is the average of $[W]$ in a ball of center $z$ and radius $r$:
\[
\Upsilon ^W_r  = [W]*\frac{\mathbf{1}_{B(0,r)}}{\vol(B(0,r))},
\]
where $\mathbf{1}_A$ denotes the characteristic function of a set $A$ and $*$ is convolution. 
\red 
\end{rmk}

A useful concept in the study of interpolation and sampling for a smooth hypersurface with respect to strictly plurisubharmonic weights is that of density of the hypersurface.  The definition given in \cite{osv}, which we now recall, extends immediately to the setting of possibly singular hypersurface.

\begin{defn}\label{d-def}
Assume $\vp$ is strictly plurisubharmonic.  The number 
\[
D_r(W;z) := \sup \left \{  \frac{\Upsilon _r ^W (z)(v,v)}{\ii \di \dbar \vp _r (z)(v,v)}\ ;\ v\in T_{\C ^n, z}-\{0\}\right \}
\]
is called the density of $W$ in the ball of radius $r$ and center $z$.  The upper density of $W$ is
\[
D^+_{\vp}(W) := \limsup _{r \to \infty} \sup _{z \in \C ^n} D_r(W;z).
\]
The lower density of $W$ is
\[
D^-_{\vp}(W) := \liminf_{r \to \infty}\inf_{z \in \C ^n} D_r(W;z).
\]
(We will not use $D^-_{\vp}(W)$ in the present article.)
\red
\end{defn}
 
As stated, the upper and lower densities are well-defined only for strictly plurisubharmonic functions.  However, one can reformulate the definition as follows. 
 \[
D^+_{\vp}(W))= \inf \left \{ \alpha  \ge 0 \ ;\  \ii \di \dbar \varphi _r - \tfrac{1}{\alpha}\Upsilon _r^W \ge 0 \text{ for all }r>>0 \right \}
\]
and
\[
D^-_{\vp}(W)= \sup \{ \gamma \ge 0\ ;\ \gamma  \ii \di \dbar \varphi _r(z) - \Upsilon _r^W(z) \not \ge 0 \text{ for all }z \in \C ^n \text{ and all }r>>0\}.
\]
These equivalent formulations of the upper and lower densities make sense for $\vp$ that are plurisubharmonic but not necessarily strictly plurisubharmonic.

Thus for example, $D^+_{\vp}(W) < 1$ if and only if there exists a constant $\delta > 0$ such that for all $r >> 0$,
\[
\ii \di \dbar \vp _r \ge (1+\delta) \Upsilon ^W_r.
\]

\section{Proof of Theorem \ref{suff-analytic}}

Part of our proof of Theorem \ref{suff-analytic} resembles the method of proof used in \cite{osv}.  The twisted Bochner-Kodaira technique used in \cite{fv} cannot be used directly to prove Theorem \ref{suff-analytic} when $W$ is not smooth.  We will illustrate this claim in Paragraph \ref{smooth-case}.

\subsection{A singular function}

As is usual in the $L^2$ approach of extension, we need to produce a function that is singular on $W$.  The function we choose is a tried-and-true one (see \cite{osv,fv}), but we contribute one new insight to the definition.

\begin{defn}\label{sing-fn-defn}
We define the function 
\[
s_r(z) = \log |T(z)|^2 - \dashint _{B(z,r)} \log |T(\zeta)|^2 \frac{\omega ^n}{n!}
(\zeta),
\]
where  $T \in \co (\C ^n)$ be a holomorphic function such that $W = \{ T=0 \}$ and $dT$ is not identically zero on $W$.
\red \end{defn}

\begin{rmk}
Note that the function $s_r$ depends only on $W$, and not on the generator $T$ of the ideal $\si _W$ of germs of holomorphic functions vanishing on $W$.
\red
\end{rmk}

Definition \ref{sing-fn-defn} and the Poincar\'e-Lelong Identity yield the following proposition.
\begin{prop}\label{t-rep}
Let $s_r$ be as in Definition \ref{sing-fn-defn}.  Then 
\[
\frac{\ii}{2\pi} \di \dbar s_r = [W]-[W]*\frac{\mathbf{1} _{B(0,r)}} {\vol(B(0,r))} = [W]-\Upsilon ^W_r.
\]
\end{prop}

\noi The proof can be found in \cite{osv} where the following lemma was also established.

\begin{lem}\label{s-prop}
The function $s_r$ has the following properties.
\begin{enumerate}
\item[(a)]  It is non-positive.

\item[(b)]  For each $r,\ve > 0$ there is a constant $C_{r,\ve}$ such that if $\dist(z, W) \ge \ve$, then $s_r(z) \ge - C_{r,\ve}$.

\item[(c)]  The function $e^{-s_r}$  is not integrable at any open subset that intersects $W$.
\end{enumerate}
\end{lem}

\begin{rmk}
The function $s_r$ is the logarithm of the length of a defining section $T$ of the (trivial) line bundle associated to $W$, measured with a metric constructed from $T$, namely $e^{-\psi _T}$, where
\[
\psi _T (z) := \dashint _{B(z,r)} \log |T(\zeta)|^2 \frac{\omega ^n(\zeta)}{n!}.
\]
Note that the curvature of $e^{-\psi _T}$ is $\ii \di \dbar \psi _T = \Upsilon ^W_r$, which depends only on $W$.
\red
\end{rmk}

\begin{rmk}
We have not missed the dependence of $\psi _T$ on the radius $r$.  In our work, this dependence will be irrelevant as soon as $r$ is sufficiently large, but how large an $r$ is needed depends on $W$.  Perhaps this dependence is important in other considerations.
\red
\end{rmk}

\subsection{The smooth case}\label{smooth-case}  

In \cite{v-tak}, the following $L^2$-extension theorem was proved.

\begin{d-thm}\label{ot-thm}
Let $X$ be a Stein manifold with K\"ahler form $\omega$, $Z \subset X$ a smooth hypersurface, $e^{-\eta}$ a singular Hermitian metric for the holomorphic line bundle associated to the smooth divisor $Z$, and $T$ a holomorphic section of this line bundle such that $Z= \{T=0\}$.  Assume that $e^{-\eta}|_Z$ is still a singular Hermitian metric, and that 
\[
\sup _X |T|^2 e^{-\eta} = 1.
\]
Let $H \to X$ be a holomorphic line bundle with singular Hermitian metric $e^{-\kappa}$ whose curvature $\ii \di \dbar \kappa$ is non-negative in the sense of currents.  Suppose also that 
\[
\ii \di \dbar \kappa + {\rm Ricci}(\omega) \ge (1+\delta) \ii \di \dbar \eta
\]
for some positive number $\delta$.  Then for each section $f \in H^0(Z, H)$ satisfying 
\[
\int _{Z} \frac{|f|^2 e^{-\kappa}}{|dT|^2 e^{-\eta}} \frac{\omega ^{n-1}}{(n-1)!} < +\infty
\]
there is is a section $F \in H^0(X,H)$ such that 
\[
F|_Z = f \quad \text{and}\quad \int _{X} |F|^2 e^{-\kappa} \frac{\omega ^n}{n!} \le \frac{C}{\delta} \int _Z  \frac{|f|^2 e^{-\kappa}}{|dT|^2 e^{-\eta}} \frac{\omega ^{n-1}}{(n-1)!},
\]
where the constant $C$ is universal.
\end{d-thm}

\begin{rmk}
Theorem \ref{ot-thm} can be easily extended to the case of singular $Z$ with essentially the same proof, provided that integration over $Z$ is replaced by integration over $Z_{\rm reg}$.  
\red
\end{rmk}

Theorem \ref{ot-thm} looks rather similar to Theorem \ref{suff-analytic}, except for the denominator $|dT|^2e^{-\eta}$ used on the subvariety $Z$.  In fact, let us take $X = \C ^n$, $Z=W$, $\omega = \tfrac{\ii}{2} \di \dbar |z|^2$, $\kappa = \vp$ and 
\[
\eta(z) := \psi _T(z) = \dashint _{B(z,r)} \log |T|^2 \omega ^n.
\]
If we define\footnote{Note that $\rho _r$ is not differentiable at $W$, but $|\di \rho _r|^2$ is well-defined on $W$ and therefore on $\C ^n$.}
\[
\rho _r :=e^{\tfrac{1}{2} s_r}
\]
and
\[
\ch (W,\vp) := \left\{ f \in \co (W)\ ;\ \int _{W_{\rm reg}}\frac{|f|^2e^{-\vp}}{|\di \rho _r|^2} \omega ^{n-1} < +\infty\right \},
\]
then we have the following theorem.

\begin{d-thm}\label{suff-anal-sing}
Let $W$ be a singular hypersurface in $\C ^n$ and $\vp$ a plurisubharmonic function in $\C ^n$.  Suppose $D^+_{\vp}(W) < 1$.  Then the restriction map $\cR _W : \sh (\C ^n,\vp) \to \ch (W,\vp)$ is surjective.
\end{d-thm}

\begin{prob}
Is the converse of Theorem \ref{suff-anal-sing} true under the additional assumption \eqref{gbf-weight}?
\end{prob}

As the next lemma shows, Theorem \ref{suff-analytic} follows from Theorem \ref{ot-thm} when $W$ is smooth (and uniformly flat), and in fact the latter is more general than the former in the case of smooth $W$, since the curvature hypotheses on $\vp$ are weaker.  

\begin{lem}\label{flat-lem}
If $W$ is smooth and uniformly flat then there is a constant $C_r$ such that 
\[
\inf _{x \in W} |\di \rho _r(x)|^2 \ge C_r.
\]
\end{lem}

\begin{proof}
We choose a point $z \in W$, and we will show that there is a lower bound for $|\di \rho _r(z)|^2$ that does not depend on $z$.  To this end, let us first fix $T \in \co (W)$ such that $W= \{ T=0\}$ and $dT|_W$ is never zero.  

Now, one can write 
\begin{eqnarray*}
|\di \rho _r(z)|^2 &=& |dT(z)|^2 \exp \left ( - \dashint _{B(z,r)} \log |T|^2 \omega ^n \right ) \\
&=& |dT(z)|^2 \exp \left ( - \dashint _{B(z,a)} \log |T|^2 \omega ^n \right ) \times \exp \left ( \left ( \dashint _{B(z,a)} - \dashint _{B(z,r)} \right ) \log |T|^2 \omega ^n \right ).
\end{eqnarray*}
The two factors in the last line are both independent of the choice of function $T$ that cuts out $W$.  For the first factor, we need only use a function that cuts out $W$ in $B(z,a)$, while for the second factor we may use a function that cuts out $W$ in $B(z,r)$.  We make choices for each factor, so as to obtain universal lower bounds.

Let us begin with the left factor, which takes place over $B(z,a)$.  By using the function $T= y - f(x)$ given by Proposition \ref{osv-uf}(G) representing the graph of $W$ near the point $z$ in question we get a uniform lower bound for the first factor.  To have such a defining function, it suffices by the uniform flatness hypothesis to take $a$ sufficiently small but independent of $z$.

Let us now turn to the second factor.  Of course, this factor is bounded above by $1$, because of the increasing property of (pluri)subharmonic averages, but we are interested in a lower bound.  To obtain the latter, we proceed as follows.  Consider the closed positive $(1,1)$-current  
\[
\Upsilon ^W_a (x) := \frac{\ii}{2\pi} \di \dbar \dashint _{B(0,a)} \log |T(\zeta + x)|^2 \omega ^n(\zeta) = [W]* \frac{\mathbf{1}_{B(0,a)}}{{\rm Vol}(B(0,a))}(x).
\]
The trace of $\Upsilon ^W_a(x)$ is ${\rm Area}(W \cap B(x,a))$, and therefore by Proposition \ref{osv-uf}(A), $\Upsilon ^W_a$ is bounded above by a multiple of the Euclidean metric, the multiple depending only on $a$.  Below we are going to consider balls of the form $B(x,a)$ as we let $x$ vary in $B(z,r)$, and therefore we want to work in $B(z,2r)$ for the moment.  From the proof of Lemma \ref{quimbo-trick} (for example, in \cite{lindholm}) we deduce that for $r > 0$ there is a (plurisubharmonic) function $u = u_{z,a,r}$ such that $\ii \di \dbar u = \Upsilon ^W_a$ in $B(z,2r)$ and  
\[
\sup _{B(z,r+2a)} |u| \le A_{a,r},
\]
where $A_{a,r}$ is independent of $z$.  

Now, the function $h(x) := u(x) - \dashint_{B(x,a)} \log |T|^2 \omega ^n$ is pluriharmonic in $B(z,2r)$, and therefore with $H \in \co (B(z,2r))$ such that $h:= 2\re H$, we have  
\[
u(x) = \dashint _{B(x,a)} \log |Te^{H}|^2 \omega ^n, \quad x\in B(z,r+a).
\]
Letting $T_o := T e^H$, we have 
\[
\left | \dashint _{B(x,a)}\log |T_o|^2 \omega ^n \right | \le A_{a,r}, \qquad x\in B(z,r+a).
\]

By the sub-mean value property, we have 
\[
\log |T_o(x)|^2 \le \dashint _{B(x,a)} \log |T_o|^2 \omega ^n \le A_{a,r},
\]
and thus setting $T_1 := T_o e^{ - \tfrac{1}{2} A_{a,r}}$ we have a function $T_1 \in \co (B(z,r+a))$ such that $\log |T_1(x)|^2 \le 0$ for all $x\in B(z,r+a)$, and therefore  
\[
\dashint _{B(x,r)}- \log |T_1|^2 \ge 0.
\]
Moreover, 
\[
\dashint _{B(z,a)} \log |T_1|^2 \omega ^n \ge - \left | \dashint _{B(z,a)} \log |T_1|^2 \omega ^n \right | \ge - 2A_{a,r}.
\]
The proof is complete.
\end{proof}

\begin{rmk}
Lemma \ref{flat-lem} clearly fails for non-smooth $W$.  It is for this reason that we said we could not use the method of \cite{fv} to prove Theorem \ref{suff-analytic} in the non-smooth case.  Even more, when $W$ is singular it is in fact the case that the spaces $\ch (W,\vp)$ and $\fH(W,\vp)$ are different.  Thus the two results \ref{suff-analytic} and \ref{suff-anal-sing} discuss extension from two completely different spaces of holomorphic functions.
\red
\end{rmk}

Notice that our proof of Theorem \ref{suff-analytic} when $W$ is smooth requires only that the weight $\vp$ be plurisubharmonic, and does not need the stronger positivity hypothesis \eqref{gbf-weight}.  On the other hand, in the proof of Theorem \ref{suff-analytic} for the general case, we will use the strong curvature hypothesis \eqref{gbf-weight}.  We see no reason at the moment why Theorem \ref{suff-analytic} cannot be extended to the case of weights that are only plurisubharmonic, and thus state the following conjecture.

\begin{conj}\label{weak-gbf}
Theorem \ref{suff-analytic} holds for any plurisubharmonic weight $\vp$ such that $D^+_{\vp}(W) <1$.
\end{conj}

We now turn our attention to the singular case.  As already indicated, we will use the approach, taken in \cite{osv}, of locally extending the data from $W$, and then patching together the local extensions using H\"ormander's Theorem.

\subsection{Local extensions}

Let $f \in \fH (W,\vp)$ be the function to be extended.

\begin{d-thm}\label{local-extensions-with-bounds}
There exists a covering of $W$ by a locally finite collection of balls $\{ B_{\alpha}\}$ each having radius $r \in [\ve _o/2, (a+1)\ve _o]$, with the following additional properties.
\begin{enumerate}
\item There is a number $N$ having the property that each point of $\C ^n$ is contained in at most $N$ balls.  
\item Write $f _{\alpha} = f|_{W \cap B_{\alpha}}$.  Then for each ${\alpha}$, there exists $F_{\alpha} \in \co (B_{\alpha})$ such that $F_{\alpha} |_{W\cap B_{\alpha}} = f_{\alpha}$ and 
\[
\int _{B'_{\alpha}}|F_{\alpha}|^2 e^{-\vp} \omega ^n \le C \int _{W_{\rm reg}\cap B_{\alpha}} |f|^2 e^{-\vp} \omega ^{n-1},
\]
where $B'_{\alpha}$ is a ball with the same center as $B_{\alpha}$ and with radius $\lambda r$ for some $\lambda \in (0,1)$ independent of $\alpha$, and the constant $C$ is independent of $j$.
\end{enumerate}
\end{d-thm}

We now embark on the proof of Theorem \ref{local-extensions-with-bounds}.  We begin by dispensing with the uniformity of the number $N$ of open balls containing any point.  To this end, it is clear that locally such a number $N$ obviously exists, and by uniform flatness the picture locally is the same everywhere.

Next, observing that by the uniform flatness hypothesis we can cover $W$ by open balls $B_{\alpha}$ of a fixed radius (which we may shrink a few times below) such that for each $\alpha$, $W \cap B_{\alpha}$ is a finite union of smooth hypersurfaces cut out by holomorphic functions $T_1,...,T_k$ where $k=k_{\alpha} \le N$ for some positive integer $N$ independent of $\alpha$.  By the definition of uniform flatness, particularly Property (S) of Definition \ref{u-flat-defn}, as well as Property (G) of Lemma \ref{flat-properties}, we may assume, perhaps after decreasing the radius of the balls $B_{\alpha}$ if necessary, that 
\begin{equation}\label{cauchy-est-tj-bounds}
\frac{1}{C} \le |dT_j(z)| \le C, \quad z\in B_{\alpha}
\end{equation}
for some constant $C > 0$ independent of $\alpha$ and $j$.   Indeed, since the $T_j$ are holomorphic, the result follows from a simple application of the Cauchy estimates to the function whose graph is cut out by $T_j$.

Uniform flatness also means that the angles between the unit vectors orthogonal to the $W_i$ at the origin are uniformly bounded away from zero (the uniformity being with respect to ${\alpha}$).  For our purposes, this property takes the following form: there exists a constant $\tilde C$ independent of the center of $B_{\alpha}$, such that for all $1 \le i \neq j \le k$,  
\begin{equation}\label{dt-lb}
 \left | dT_j(v) \right | \ge \tilde C |v|  \quad \text{for all } v\in T^{1,0}_{W_i} -\{0\}.
\end{equation}

We have the following lemma.

\begin{lem}\label{vanish=divide}
Define the variety
\[
W_{\alpha} := \bigcup _{1 \le j \le k} \{T_j=0\} \subset B_{\alpha}.
\]
Suppose $0 \in W_{\alpha}$ is a singular point, and that $F \in \co (B)$ vanishes on $W_{\alpha}$.  Then for any multiindex $I=(i_1,...,i_{j-1})$ such that $1 \le i_{\ell} \le k$, $F$ is divisible by the product $T_{i_1}\cdot ...\cdot  T_{i_{j-1}}$.
\end{lem}

\begin{proof}
It suffices to prove that 
\[
\frac{F}{T_1\cdots T_{k}} \in \co (B_{\alpha}).
\]
Indeed, if $\frac{F}{T_1\cdots T_k}$ is holomorphic, and we could then eliminate any undesired factor $T_{\ell}$ in the denominator simply by multiplying by $T_{\ell}$.

To prove the holomorphicity of the latter, observe that $\frac{F}{T_1\cdots T_{k}}$ is holomorphic at all the smooth points of $W_{\alpha}$, i.e., the points of $W_i$ where $T_j \neq 0$ for all $j \neq i$.  Since $\frac{F}{T_1\cdots T_{k}}$ is clearly holomorphic away from $W_{\alpha}$, it follows that the poles of $\frac{F}{T_1\cdots T_{k}}$ are contained in the set of points of $W_{\alpha}$ where at least two of the $T_j$ vanish.  But since any two branches of $W_{\alpha}$ intersect transversally, the polar set of $\frac{F}{T_1\cdots T_{k}}$ is contained in a subvariety of codimension at least $2$.  On the other hand, the polar set of $\frac{F}{T_1\cdots T_{k}}$ is a divisor, and therefore it must be empty.  
\end{proof}

For ease of notation, we now drop the subscript $\alpha$ and work on a fixed ball $B$, whose center we may assume, without loss of generality, is the origin.  

\begin{lem}\label{denom-clear}
Let $g :W_j \to \C$ be a holomorphic function such that  $\frac{g}{T_1\cdot \dots T_{j-1}} \in \co (W_j)$.  Then there exist positives constant $C < \frac{1}{\ve_o}$ and $\widehat C$ independent of the center of $B$ such that, with $B^{k,\ve_o}$ denoting the ball whose center is the center of $B$ and whose radius is $(1-C\ve_o)^k$ times that of $B$, 
\begin{equation}\label{local-est-for-denom-clear}
\int _{W_j \cap B^{j-1,\ve_o}} \frac{|g|^2e^{-\vp}}{|T_1 \cdots T_{j-1}|^2} \omega ^{n-1} \le \widehat C \int _{W_j} |g|^2 e^{-\vp} \omega ^{n-1}.
\end{equation}
\end{lem}

\begin{proof}
By Lemma \ref{quimbo-trick} we may assume, upon replacing $g$ by $ge^{G}$ for an appropriate holomorphic function $G$ defined on a neighborhood of the closure of $B$, that $\vp = 0$.   

Suppose first that $j=2$.  We can assume without loss of generality (perhaps after slightly shrinking the ball $B$ at the outset by a small factor independent of the center of $B$) that $W_2$ is the unit ball with coordinates given by $T_1,z^2,...,z^{n-1}$.  Uniform flatness renders all of the scaling uniform in the center of $B$.   Let $P \in B$ satisfy $|P|=1-2\ve_o$.  If $|T_1(P)| \ge \frac{\ve _o}{2^{n+1}}$ then \[
\frac{|g(P)|}{|T_1(P)|} \le \frac{2^{n+1}}{\ve _o} |g(P)|.
\]
On the other hand, if $|T_1(P)| < \frac{\ve _o}{2^{n+1}}$ then by the Cauchy formula we have 
\begin{eqnarray*}
\left |\frac{g(P)}{T_1(P)}\right | &=& \frac{1}{2\pi} \left | \int _{|T_1-T_1(P)|= \frac{\ve_o}{2^n}}\frac{g(T_1,z^2(P),...,z^{n-1}(P))}{T_1(T_1-T_1(P))}dT_1 \right |\\
&\le &\frac{1}{2\pi}\int _{|T_1-T_1(P)|= \frac{\ve_o}{2^n}} \frac{|g(T_1,z^2(P),...,z^{n-1}(P))|}{|T_1-T_1(P)|^2\left (1-\frac{|T_1(P)|}{|T_1-T_1(P)|}\right )} |dT_1|\\
&\le & \frac{1}{2\pi}\int _{|T_1-T_1(P)|= \frac{\ve_o}{2^n}} \frac{2 |g(T_1,z^2(P),...,z^{n-1}(P))|}{|T_1-T_1(P)|}\frac{|dT_1|}{|T_1-T_1(P)|}\\
&\le & \frac{2^{n+1}}{\ve _o} \sup  \{ |g(T_1(Q),z^2(P),...,z^{n-1}(P)))|\ ;\ |T_1(Q)-T_1(P)| \le 2^{-n}\ve _o\}.
\end{eqnarray*}
Thus 
\[
\frac{|g(P)|}{|T_1(P)|} \le \frac{2^{n+1}}{\ve _o}\sup  \{ |g(Q)|\ ;\ \sqrt{|T_1(Q)-T_1(P)|^2+|z'(Q)-z'(P)|^2} \le 2^{-n}\ve _o\},
\]
where $z'=(z^2,...,z^{n-1})$.  By the maximum principle, 
\[
\sup _{B^{1,2\ve_o} \cap W_2} \frac{|g|}{|T_1|} \le \sup _{B^{1,\ve _o}\cap W_2} |g|.
\]
From these sup-norm estimates the result easily follows, and we have the case $j=2$.  If we write 
\[
\frac{g}{T_1\cdots T_{j-1}} = \frac{g/(T_2\cdots T_{j-1})}{T_1},
\]
the remaining cases follow by induction.
\end{proof}

\begin{proof}[Proof of Theorem \ref{local-extensions-with-bounds}]
Let 
\[
f_i := f|_{W_i}.
\]
Consider first the function $f_1$.  By \eqref{dt-lb} we have
\[
\int _{W_1} \frac{|f_1|^2e^{-\vp}}{|dT_1|^2} \omega ^{n-1} \lesssim \int _{W_1} |f_1|^2e^{-\vp} \omega ^{n-1} < +\infty.
\]
By Theorem \ref{ot-thm} with $\eta = 0$ and $\omega$ the Euclidean metric in $\C ^n$, there exists a holomorphic extension $F_1$ of $f_1$ to $B$ such that 
\[
\int _{B} |F_1|^2e^{-\vp} \omega ^n \lesssim \int _{W_1} |f_1|^2e^{-\vp} \omega ^{n-1}.
\]
Let $F^1 := F_1$, and notice that 
\[
F^1|_{W_1} = f_1 \quad \text{and} \quad \int _B |F_1|^2 e^{-\vp} \omega ^n \lesssim \int _{W\cap B} |f|^2e^{-\vp} \omega ^{n-1}.
\]

We now argue by induction.  Suppose we have found $F^{j-1} \in \co (B)$ such that 
\[
F^{j-1}|_{W_i} = f_i \quad \text{for} \quad 1 \le i \le j-1
\]
and 
\[
\int _{B} |F^{j-1}|^2 e^{-\vp} \omega ^n \lesssim \int _{B\cap W} |f|^2e^{-\vp} \omega ^{n-1}.
\]

Consider the function 
\[
f_j^* := f_j - F^{j-1}|_{W_j}.
\]
We observe that for each $1\le i \le j-1$, 
\begin{equation}\label{vanish-on-i}
f_j^* |_{W_i} = f_j|_{W_i} - f_i |_{W_j} = 0.
\end{equation}
The last equality follows because $f$ is assumed to have a local extension to a neighborhood of $W \cap B$ in $B$.  It follows from Lemma \ref{vanish=divide} that 
\begin{equation}\label{quotient-fn}
\frac{f_j^*}{\prod _{i=1} ^{j-1} T_i} \in \co (W_j - ((W_1\cap W_j) \cup  ... \cup (W_{j-1} \cap W_j)))
\end{equation}
extends holomorphically to $W_j$, and thus by Lemma \ref{denom-clear} the extension satisfies the estimate 
\begin{equation}\label{claim-est}
\int _{W_j\cap B^{j,\ve_o}} \frac{|f^*_j|^2e^{-\vp}}{|T_1\cdot \dots \cdot T_{j-1}|^2{|dT_j|^2}} \omega ^{n-1}  \lesssim \int _{W_j} |f_j^*|^2e^{-\vp} \omega ^{n-1} < +\infty.
\end{equation}
By Theorem \ref{ot-thm} there exists a holomorphic function $F_j^* \in \co (B^{j,\ve})$ such that 
\[
F_j^* |_{W_j} = f^*_j \quad \text{and} \quad \int _{B^{j,\ve}} \frac{|F_j^*|^2e^{-\vp}}{|T_1 \cdot \dots \cdot T_{j-1}|^2} \omega ^n \lesssim \int _{W} |f|^2e^{-\vp} \omega ^{n-1}.
\]
In particular, $F_j^*|_{W_i} = 0$ for all $1 \le i < j$.  Let 
\[
F^j := F^{j-1} + F^*_j.
\]
Then for $1 \le i < j$,  
\[
F^j|_{W_i} = F^{j-1}|_{W_i} = f_i.
\]
moreover, 
\[
F^j|_{W_j} = F^{j-1}|_{W_j} + F_j ^*|_{W_j} = f_j.
\]
Finally, 
\begin{eqnarray*}
\int _{B^{j,\ve}} |F^j|^2 e^{-\vp} \omega^n  & \lesssim & \int _{B^{j-1,\ve}} |F^{j-1}|^2 e^{-\vp} \omega ^n + \int _{B^{j,\ve}} |F^*_j|^2 e^{-\vp} \omega ^n \\
& \lesssim & \int _{B^{j,\ve}} |F^{j-1}|^2 e^{-\vp} \omega ^n + \int _{B^{j,\ve}} \frac{|F^*_j|^2 e^{-\vp}}{|T_1 \cdot \dots \cdot T_{j-1}|^2} \omega ^n \\
& \lesssim & \int _{W\cap B} |f|^2 e^{-\vp} \omega ^{n-1}.
\end{eqnarray*}
By induction on $j$ we obtain the existence of a function $F:= F^k \in \co (B)$ which evidently satisfies the desired conclusions.  This completes the proof.
\end{proof}

Now that we have found our local extensions with good bounds, we patch them together.

\subsection{The patching process}\label{patching}

We begin with the balls $B_{\alpha}$ and functions $F_{\alpha}$ of Theorem \ref{local-extensions-with-bounds}.  We can assume that our open cover $\{B_{\alpha}\}$ is such that 
\[
B_{\alpha} = B(w_{\alpha} ,2a_{\alpha}\ve) \quad \text{and} \quad \bigcup _{j} B(w_{\alpha},a_{\alpha}\ve) \supset W.
\]
Here $a_{\alpha}=a$ if $w_{\alpha} \in W_{\rm sing}$ and $a_{\alpha} = 1$ if $w_{\alpha} \in W_{\rm reg}$.  For simplicity of exposition, let 
\[
\widehat B_{\alpha} := B(w_{\alpha},a_{\alpha}\ve )
\]
be the notation for the `half-balls'.  We assume the index ${\alpha}$ begins at $1$, and add the set 
\[
B_o = \widehat B_0 := \C ^n - \left (U_{\ve}(W) \cup U_{a\ve}(W_{\rm sing})\right )
\]
to the open cover, to obtain an open cover of $\C ^n$.  We let $F_0 = 0$.  Then we fix a partition of unity $\{ \phi _{\alpha} \}$ subordinate to the cover $\{ \widehat B_{\alpha}\}$, which we will assume has the property 
\[
\sum _{\alpha} |d \phi _{\alpha}|^2 \le C.
\]
We seek a global holomorphic function $F$ such that 
\[
F|_W= f \quad \text{and} \quad \int _W |F|^2 e^{-\vp} \omega ^n < +\infty.
\]
To this end, let 
\[
G_{\alpha\beta} = F_{\alpha} - F_{\beta} \in \co (\widehat B_{\alpha} \cap \widehat B_{\beta}).
\]
Let us set 
\begin{equation}\label{psi-choice}
\psi := \vp _r + s_r.
\end{equation}
We have
\[
G_{\alpha \beta}|_{W\cap \widehat B_{\alpha} \cap \widehat B_{\beta}} \equiv 0 \quad \text{and} \quad \int _{\widehat B_{\alpha} \cap \widehat B_{\beta}} |G_{\alpha \beta}|^2 e^{-\psi} \omega ^n \lesssim \int _{W_{\rm reg} \cap B_{\alpha} \cap B_{\beta}}|f|^2 e^{-\vp} \omega ^{n-1}.
\]
The vanishing of the restriction is obvious, and the inequality is established in exactly the same way as in the proof of Lemma 4.3 in \cite{osv}.  We also take this opportunity to observe that 
\begin{equation}\label{pre-hormander}
\ii \di \dbar \psi \ge  \ii \di \dbar \vp _r - \Upsilon ^W_r = \delta ' \ii \di \dbar \vp _r + (1-\delta') \vp _r - \Upsilon^W_r \ge \delta ' \omega
\end{equation}
for some $\delta' >0$ sufficiently small.

We claim that there are functions $G_{\alpha} \in \co (\widehat B_{\alpha})$ such that 
\begin{equation}\label{cousin-correction}
G_{\alpha} |_{W \cap B_{\alpha}} \equiv 0, \quad \int _{\widehat B_{\alpha}}|G_{\alpha}|^2 e^{-\vp} \omega ^n \lesssim \int _{W_{\rm reg}\cap  B_{\alpha}} |f|^2 e^{-\vp} \omega ^{n-1} \quad \text{and} \quad G_{\alpha} \in \co (\widehat B_{\alpha}).
\end{equation}
The functions 
\[
\tilde G_{\alpha} := \sum _{\beta} G_{\alpha \beta} \phi _{\beta}
\]
have the first two of the properties \eqref{cousin-correction}.  Moreover, we have 
\[
\dbar (\tilde G_{\alpha} - \tilde G_{\beta}) = \dbar G_{\alpha \beta} = 0 \quad \text{on }\widehat B_{\alpha} \cap \widehat B_{\beta},
\]
and thus we can define the global $\dbar$-closed $(0,1)$-form 
\[
H = \dbar \tilde G_{\alpha} = \sum _{\beta} G_{\alpha \beta} \dbar \phi _{\beta} \quad \text{on }\widehat B_{\alpha}.
\]
We calculate that
\begin{eqnarray*}
\int _{\C ^n} |H|^2 e^{-\psi} \omega ^n &\lesssim & \sum_{\alpha \beta} \int _{\widehat B_{\alpha} \cap \widehat B_{\beta}} |G_{\alpha \beta}|^2 |d\phi _{\beta}|^2 \omega ^n \\
&\lesssim & \sum_{\alpha \beta} \int _{W_{\rm reg} \cap B_{\alpha} \cap B_{\beta}} |f|^2 e^{-\vp} \omega ^{n-1} \\
&\lesssim & \int _{W_{\rm reg}} |f|^2 e^{-\vp} \omega ^{n-1}.
\end{eqnarray*}
Since the right side is finite, H\"ormander's Theorem (which, in view of \eqref{pre-hormander}, may be used) provides a function $u$ satisfying
\[
\dbar u = H \quad \text{and} \quad \int _{\C ^n} |u|^2 e^{-\vp}\omega ^{n-1} \le \int _{\C ^n} |u|^2 e^{-\psi}\omega ^{n-1} < +\infty
\]
The second estimate implies that $u|_W\equiv 0$.  The functions 
\[
G_{\alpha} := \tilde G_{\alpha} - u
\] 
are therefore holomorphic and satisfy \eqref{cousin-correction}, and the function 
\[
F:= F_{\alpha} - G_{\alpha} \quad \text{on }\widehat B_{\alpha}
\]
satisfies 
\[
F|_W = f \quad \text{and} \quad \int _{\C ^n}|F|^2 e^{-\vp} \omega ^n < +\infty.
\]
The proof of Theorem \ref{suff-analytic} is thus complete.
\qed

\section{Non-uniformly flat hypersurfaces may have extension}\label{non-necess}

In this last section, we give a negative answer to the question of whether uniform flatness is necessary for the surjectivity of $\sr _W$.  In fact, there are two cases one should treat separately.  The first case is that of singular hypersurfaces, and the second that of smooth hypersurfaces.

\subsection{Non-necessity for singular hypersurfaces}

The content of this section is an observation of Ohsawa \cite{o-ober} to the effect that the restriction map $\sr _W : \sh (\C ^n, \vp) \to \fH (W,\vp)$ may be surjective even if $W$ is not uniformly flat.  In fact, let $W = \{ T=0\}$ be a hypersurface with an isolated singularity at $0$, i.e., 
\[
T(z)= dT(z)= 0 \iff z=0.
\]
Let $f \in \fH (W, \vp)$.  Then there is a polynomial $P \in \C [z^1,...,z^n]$ such that 
\[
\int _{W_{\rm reg}} \frac{|f-P|^2}{|dT|^2}e^{-\vp} \omega ^{n-1} < +\infty.
\]
Assume now that 
\[
D^+_{\vp}(W) < 1 \quad \text{and} \quad \int _{\C ^n} |P|^2 e^{-\vp} \omega ^n < +\infty
\]
(which holds, for example, if $W$ is algebraic and $\vp(z) = |z|^2$).  Then Theorem \ref{suff-anal-sing} implies that there is a function $F_o \in \sh (\C ^n,\vp)$ such that 
\[
F_o |_W = f- P|_W.
\]
It follows that $F:= F_o+P$ is the desired extension.

\subsection{Non-necessity for smooth hypersurfaces}

The main result of this section is the following theorem.

\begin{thm}\label{non-necess-smooth}
The embedded smooth curve  
\[
W:= \{ (x,y) \in \C ^2\ ;\ xy\sin y = 1\} \emb \C ^2
\]
is not uniformly flat, but nevertheless $\sr _W : \sh (\C ^2, |\cdot |^2) \to \fH (W,|\cdot |^2)$ is surjective.
\end{thm} 

Near the points $(0,2\pi n)$ for $n \in \Z$ with $|n| >>0$, the curve $W$ looks a lot like the curve 
\[
V_n := \{ (x,y)\ |\ x(y-2\pi n) = \frac{1}{2\pi n}.
\]
From this approximation of $W$ by such model curves, it follows immediately that $W$ is not uniformly flat.  On the other hand, the model curve $V_n$ `converges', as $|n|\to\infty$, to a singular curve which, by Theorem \ref{suff-analytic}, induces a surjective restriction map.  Thus if we can establish some continuity in the extension process, we will be able to prove the claimed surjectivity of $\sr _W$.

The approach is to begin by establishing the desired continuity in the model case, and then patching together perturbations of the model extensions.

\subsubsection{\sf The model case}

The continuity of the extension process we seek in this model case may be expressed by saying that there is an extension operator $\se _{s} : \fH (W_s, |\cdot |^2) \to \sh (\C ^2, |\cdot |^2)$ whose square norm 
\[
||\se _s||^2 = \sup \left \{ \int _{\C ^2}|\se _sf|^2 e^{-|\cdot |^2}\omega ^2\ \left | \ \int _{W_s}|f|^2 e^{-|\cdot |^2}\omega=1 \right \} \right .
\]
is bounded independent of $s$.  The extension operator in question is going to be the extension of minimal norm.  To bound this operator, we need {\it any} extension operator with the desired bounds.

We now define an extension operator that works.  Consider the map 
\[
j : \C ^* \to W_s ; t \mapsto (t, st^{-1}).
\]
We pull back functions and integrals on $W_s$ to functions and integrals on $\C ^*$.  If we have $f \in \fH (W, |\cdot |^2)$ then with 
\[
j ^*f (t) = \sum _{j \in \Z} a_j t^j
\]
we have 
\[
||f||^2 = \int _{\C ^*} \left ( \sum _{j,k\in \Z}a_j\bar a_k  t^j \bar t^k \right )(1+ |s|^2|t|^{-4}) e^{-(|t|^2 + |s|^2 |t|^{-2})}\frac{\ii}{2} dt \wedge d\bar t= \sum _{j \in \Z} |a_j|^2 C_{j,|s|},
\]
where 
\[
C_{j,|s|} = \pi \int _0 ^{\infty}r^{2j}(1+|s|^2r^{-4}) e^{-(r^2 + |s|^2 r^{-2})}r dr.
\]
We will need lower bounds for the constants $C_{j,|s|}$.  For positive $j$, it is easy to estimate these constants, and the symmetry of the curve $W_s$ will give us a handle on the case of negative $j$.

\begin{lem}\label{laurent-constants}
The constants $C_{j,|s|}$ have the following properties.
\begin{enumerate}
\item[(i)] For $j \ge 0$ and $|s|$ sufficiently small, $C_{j,|s|} \ge  \frac{\pi}{2} (j!)$.
\item[(ii)] For all $j$ and $s$, $C_{j,|s|} = |s|^{-2j}C_{-j,|s|}$.
\end{enumerate}
\end{lem}

\begin{proof}
First, let $j \ge 0$.  Then
\begin{eqnarray*}
C_{j,|s|} &\le & \pi \int _0 ^{\infty} r^{2j+1}e^{-r^2} dr + \pi \int _0 ^{\infty} (r^{2j}e^{-r^2}) e^{-|s|^2/r^2} \frac{|s|^2}{r^3} dr\\
& \le & \pi (j!) + \frac{\pi}{2} j^je^{-j}\int _0^{\infty} e^{-u} du\\
& \le & 2\pi (j!)
\end{eqnarray*}
Thus by the Dominated Convergence Theorem, 
\[
\lim _{|s| \to 0} C_{j,|s|} = \pi (j!),
\]
and Property (i) is proved.  Property (ii) is obtained by substituting $r \mapsto |s|/r$ in the integral.
\end{proof}

We now define the extension of $f$ to be the holomorphic function $F$ given by the Taylor series
\[
F(x,y) = a_0 + \sum _{n > 0} (a_n x^n + s^{-n}a_{-n} y^n),
\]
where 
\[
j^*f (t) = \sum _{n \in \Z} a_n t^n.
\]
Note that $j^* F= j^*f$ so that $F$ extends $f$, and that 
\begin{eqnarray*}
||F||^2 &:=& \int _{\C ^2} |F(x,y)|^2 e^{-(|x|^2+|y|^2)}\frac{\ii}{2} dx \wedge d\bar x \wedge \frac{\ii}{2} dy \wedge d\bar y \\
& = & \int _0 ^{\infty}\!\!\!\! \int _0 ^{\infty}\!\!\!\! \int _0 ^{2\pi} \!\!\!\! \int _0 ^{2\pi}\!\!\!\! (a_0 + \sum _{n > 0} (a_n e^{\ii n\theta }r^n + s^{-n}a_{-n} e^{\ii n\phi} \rho ^n))\times \\
&& \qquad (\bar a_0 + \sum _{m > 0} (\bar a_m e^{-\ii m\theta} r^m + \bar s^{-m}\bar a_{-m}e^{-\ii m\phi}\rho^m)) d\theta d\phi e^{-r^2} rdr e^{-\rho ^2}\rho d\rho\\
& = & 4 \pi ^2 \int _0 ^{\infty}\!\!\!\! \int _0 ^{\infty}\!\!\!\! \left (|a_0|^2 + \sum _{n>0} |a_n|^2 r^{2n}+ |s|^{-2n}|a_{-n}|^2 \rho ^{2n} \right )e^{-r^2} rdr e^{-\rho ^2}\rho d\rho\\
& = & \pi ^2\left ( \sum _{n \ge 0} |a_n|^2 n! + \sum _{n < 0} |a_n|^2 |s|^{2n} (-n)! \right )\\
&& \quad \le  2 \pi \sum _{n \in \Z} C_{n,|s|} |a_n|^2 = 2 \pi ||f||^2.
\end{eqnarray*} 
The inequality is of course a consequence of Lemma \ref{laurent-constants}.  We have thus proved the following  lemma.

\begin{lem}\label{model-extension-estimate}
There exists a positive number $r_o$ such that for $0<|s| < r_o$, the minimal extension operator $\se _s : \fH (W_s,|\cdot |^2)\to \sh (\C ^2, |\cdot |^2)$ satisfies 
\[
||\se _s|| \le\sqrt{2 \pi}.
\]
\end{lem}

\subsubsection{\sf Clipping and translating the model}

Let $c=(a,b) \in \C^2$ and $\ve > 0$.  Define
\[
W_{s}(\ve;c) := \left \{ (x,y)\in \C^2\ ;\ (x-a)(y-b)=s \text{ and }
\max (|x-a|, |y-b|)< \ve \right \} \subset \Delta ^2 _{c}(\ve),
\]
where $\Delta ^2 _{c}(\ve)=\{(x,y) \in \C ^2\ ; \max( |x-a|, |y-b|) < \ve \}$.   We will now modify the calculation for the model to show the following.

\begin{lem}\label{clipped-model-est}
For each $f \in \co (\overline{W_s(\ve;c)})$ there exists $F \in \co (\Delta ^2_{c}(\ve))$ such that 
\[
F|_{W_s(\ve;c)}= f \quad \text{and} \int _{\Delta ^2_{c}(\ve)} |F(z)|^2 e^{-|z|^2}dV(z) \le 2\pi  \int _{W_s(\ve;c)} |f(z)|^2 e^{-|z|^2} \omega (z).
\]
\end{lem}

\begin{proof}
To begin, we translate $W_s(\ve;a)$ to the origin with the change of variables 
\[
\zeta = (\xi,\eta) = (x-a,y-b)=z-c.
\]
Then
\[
\int _{W_s(\ve;c)} |f(z)|^2 e^{-|z|^2} \omega (z) = \int _{W_s(\ve;0)} |f(\zeta+c) e^{\zeta \cdot c - \tfrac{1}{2} |c|^2} |^2 e^{-|\zeta|^2} \omega (\zeta).
\]
Thus the function 
\[
f_c(\zeta) = f(\zeta +c)e^{\zeta \cdot c - \tfrac{1}{2}|c|^2}
\]
is square integrable on the part of the model $W_s$ contained in $\Delta ^2_0(\ve)$.  An extension of this function to a function $F_c$ that is square integrable on $\Delta^2 _0(\ve)$ would then lead to the extension 
\[
F (z) = F_c(z-c)e^{-\zeta \cdot c + \tfrac{1}{2}|c|^2},
\]
and the latter is square integrable on $\Delta ^2_c(\ve)$.  Thus we have reduced to the situation $c=0$.

The remainder of the proof proceeds in a manner directly analogous to the proof of Lemma \ref{model-extension-estimate}, except that we work inside the bidisk $\Delta ^2_0(\ve)$.  $W_s(\ve;0)$ is now parameterized by an annulus $A_{\ve,|s|} = \{ |s|/\ve < |t| < \ve \}$\footnote{Note that we must have $\ve^2 > |s|$, but this must be the case if $W_s(\ve;0) \neq \emptyset$} but the parameterization is still the same map, namely 
\[
t \mapsto (t,s/t).
\]
The integrals to be estimated, instead of being over the entire plane, are now constrained to lie in $A_{\ve,|s|}$, but the integrands are rather concentrated in this annulus anyway.  We leave it to the reader to check that if 
\[
j^*f (t) = \sum _{n \in \Z} a_n t^n
\]
then 
\[
\int _{W_s(\ve;0)} |f|^2 e^{-|z|^2} = \sum _{n \in \Z} C_{n,|s|}(\ve) |a_n|^2
\]
where
\[
C_{j,|s|}(\ve) = \pi \int _{|s|/\ve} ^{\ve}r^{2j}(1+|s|^2r^{-4}) e^{-(r^2 + |s|^2 r^{-2})}r dr.
\]
Since $C_{j,|s|}(\ve)$ is increasing and bounded in $\ve$, we can apply the dominated convergence theorem to see that 
\[
\lim _{|s| \to 0} C_{j,|s|}(\ve) = C_{j,0}(\ve) := \pi e^{-\ve^2} \sum _{k=0} ^j \frac{j!}{k!} \ve ^{2k}.
\]
We therefore have the lower bound 
\[
C_{j,|s|}(\ve)  \ge 2C_{j,0}(\ve)
\]
for all sufficiently small $|s|$, and we calculate that the square of the $L^2$-norm of the extension 
\[
F(x,y) = a_0 + \sum _{n > 0} (a_n x^n +s^{-n}a_{-n} y^n)
\]
of $f$ to $\Delta ^2_0(\ve)$ is
\[
||F||^2 = \int _{\Delta ^2_0(\ve)}|F(z)|^2 e^{-|z|^2} dV(z) = \pi \left ( \sum _{n\ge 0} |a_n|^2 C_{n,0}(\ve) + \sum _{n < 0} |a_n|^2 |s|^{2n}C_{n,0}(\ve) \right ) \le 2\pi ||f||^2
\]
The proof is complete.
\end{proof}

\subsubsection{\sf The local picture near the approximate singularities of $\mathsf{W}$}

Since $W$ is the zero set of the holomorphic function $g(x,y) = xy\sin (y) - 1$, we need to examine $W$ near the points $(0,n\pi)$, $n \in \Z$, where, at least for $|n|$ large, it greatly resembles a translate of the model $W_s$ for $s = \frac{(-1)^{n+1}}{n\pi}$.

\begin{lem}\label{perturb-the-model}
There is a positive constant $\ve > 0$ and injective holomorphic maps $\psi _n$ defined on a small disk centered at the origin and satisfying $\sup _{|\zeta| < \ve} |d\zeta -(-1)^{n+1}d\psi _n(\zeta)| = O(\ve)$ uniformly in $n$ for $|n|$ sufficiently large, such that 
\[
W \cap \Delta ^2 _{(0,n\pi)}(\ve)  = \left \{ (x,y)\ ;\ x \psi _n (y- n\pi) = \tfrac{(-1)^{n+1}}{n\pi} \right \} \cap \Delta ^2 _{(0,n\pi)}(\ve).
\]
\end{lem}

\begin{proof}
Let 
\[
\psi _n (z) = \frac{(z+n\pi)\sin (z+n\pi)}{n\pi} =  \frac{(z+n\pi)(-1) ^{n+1}\sin z}{n\pi}.
\]
Then 
\[
\left | (-1)^{n+1}\tfrac{d\psi_n}{dz} -1 \right |\le \frac{|\sin z|}{|n|\pi} + \frac{|z \cos z|}{|n|\pi} + |\cos z -1|.
\]
The right side of the estimate is clearly controlled in a neighborhood $|z| < \ve$ by a constant that is uniform in $|n| >> 0$.  It follows from the proof of the implicit function theorem that, for some $\ve > 0$, $\psi _n$ is a local diffeomorphism in a neighborhood $|z| < \ve $ of $0$, and we can choose $\ve$ independent of $n$ provided $|n|$ is sufficiently large.  The proof is finished.
\end{proof}
 
\subsubsection{\sf End of the proof of Theorem \ref{non-necess-smooth}}

Let $\ve > 0$ be as in Lemma \ref{perturb-the-model}.  Outside the open set 
\[
U := \bigcup _{n \in \Z} \Delta ^2_{(0,n\pi)}(\ve)
\]
$W$ is uniformly flat in the sense of Definition \ref{u-flat-defn}.  We may assume, perhaps after shrinking $\ve > 0$ slightly, that there is an open cover of some neighborhood of $W$ in $\C ^2$ by open sets $U_j$ that are either of the form $\Delta _{(0,n\pi)}^2(\ve)$ or are balls $B_{p_j}(\ve)$ of radius $\ve$ and center $p_j$, and in the latter case $W \cap B_{p_j}(\ve)$ is the graph, over its tangent space at $p_j$, of a function bounded by a small quadratic as in Property (F1) of Lemma \ref{flat-properties}.  Moreover, we can assume that the number of elements of this cover containing any one point is bounded independent of the point.

For the neighborhoods $\Delta ^2 _{(0,n\pi)}(\ve)$ we have local extension of our datum $f$ with $L^2$-bounds independent of $n$.  Indeed, by using Lemma \ref{perturb-the-model} we may reduce to the model case, in which the claim is a consequence of Lemma \ref{clipped-model-est}.  
On the other hand, for those neighborhoods $B_p(\ve)$ we have such uniform $L^2$ extensions for the same reason as in the proof of Theorem \ref{suff-analytic}.

Finally we would like to apply the patching technique to extend $f$ to a function $F \in \co (\C ^n)$ with the estimate 
\[
\int _{\C ^n} |F|^2 e^{-|z|^2} \omega ^n \lesssim \int _W |f|^2 e^{-|z|^2} \omega ^{n-1}.
\]
This extension is done in exactly the same way as in Subsection \ref{patching}, as soon as we prove the following proposition.

\begin{prop}
For the curve $W$ of Theorem \ref{non-necess-smooth}, $D^+_{|\cdot|^2}(W) = 0$.
\end{prop}

\begin{proof}
Geometrically, the density of $W$ is the maximum, over all directions $v \in S^{3}$, of the number of intersection points of $W \cap B_p(r)$ with the average straight line $\ell$ parallel to $v$, divided by the area of the disk $\ell \cap B_p(r)$.  The average is taken over the lines.  We find that $\Upsilon ^W _r(z) = O(r^{-1})$.
\end{proof}

\noi Thus the proof of Theorem \ref{non-necess-smooth} is complete.\qed

\subsection{Something must take the place of uniform flatness}

Uniform flatness cannot be completely done away with even if the hypersurface in question is smooth.  For example, in dimension $1$ it is known to be necessary, and by taking a product, we can easily find examples of hypersurfaces with density less than $1$ such that the restriction map $\sr _W : \sh (\C ^n, \vp) \to \fH (W,\vp)$ is not surjective.  

\begin{ex}
Consider the sequence $\{z_{j}\}_{j \ge 2} \subset \C$ defined by 
\[
z_{2k} = k^2, z_{2k+1}= k^2- k^{-1}.
\]
Let 
\[
W = \{ (z_j, w) \in \C ^2\ ;\ w \in \C,\ j=2,3,4,...\}.
\]
Then the area of $W \cap B(x,r) = o(r^3)$ so that the density of $W$ is zero.  But we claim that the restriction map $\sr _W : \sh (\C ^n, \vp) \to \fH (W,\vp)$ is not bounded surjective.  

To prove this, we argue by contradiction.  Suppose, then, that $\sr _W : \sh (\C ^n, \vp) \to \fH (W,\vp)$ is bounded and surjective.  Define the locally constant function $f_k \in \co (W)$ by 
\[
f_k(z_{2k},w)= e^{\frac{1}{2}|z_{2k}|^2} \quad \text{and} \quad f(z_{j},w) = 0, \qquad j \neq 2k, \ w\in \C.
\]
Then 
\[
\int _W|f_k(z,w)|^2 e^{-(|z|^2 +|w|^2)} \omega =  \int _{\C}  e^{-|w|^2} \ii dw \wedge d\bar w = \pi  < +\infty.
\]
Since the restriction map is bounded surjective, there exists $F_k \in \co (\C ^2)$ such that 
\[
F_k|_W= f_k \quad \text{and} \quad \int _{\C ^2}|F_k|^2e^{-|z|^2-|w|^2} dV(z,w)  \le \pi C
\]
for some constant $C>0$ independent of $k$.  By the sub-mean value property,
\[
\int _{\C} |F_k(z,0)|^2e^{-|z|^2} dA(z) \le \frac{1}{\pi} \int _{\C} \left (\int _{\C} |F_k(z,w)|^2e^{-|z|^2}dA(z) \right ) e^{-|w|^2}dA(w) \le C.
\]
It follows from Cauchy's formula that, for $r \ge 2$ and $k$ sufficiently large,  
\begin{eqnarray*}
k^2 &=& \left |\frac{F(z_{2k},0)}{z_{2k}-z_{2k+1}} \right |^2e^{-|z_{2k}|^2} = \frac{1}{4\pi^2} \left | \int _{|\zeta - z_{2k}|= r} \frac{F(\zeta,0)e^{(z_{2k}^2-\zeta z_{2k})}}{(\zeta - z_{2k})(\zeta - z_{2k+1})} d\zeta \right |^2 e^{-|z_{2k}|^2}\\
&\le & \frac{e^{r^2}}{4\pi^2 r^2(r-k^{-1})^2} \left (\int _{|\zeta-z_{2k}| = r}|F(\zeta,0)|e^{\frac{1}{2} |z_{2k}|^2-\re (\zeta z_{2k}) -\tfrac{1}{2} |\zeta - z_{2k}|^2} d\theta \right )^2\\
&=&  \frac{e^{r^2}}{4\pi^2 r^2(r-k^{-1})^2} \left (\int _{|\zeta-z_{2k}| = r}|F(\zeta,0)|e^{-\frac{1}{2} |\zeta |^2} d\theta \right )^2\\
&\le & \frac{e^{r^2}}{2\pi r^2(r-k^{-1})^2} \int _{|\zeta-z_{2k}| = r}|F(\zeta,0)|^2e^{-|\zeta|^2} d\theta,
\end{eqnarray*}
where $\zeta = z_{2k} +re^{\ii \theta}$, in the second step we have used $|\zeta - z_{2k+1}| \ge |\zeta - z_{2k}|- k^{-1}$, and in the last step we have used the Cauchy-Schwarz Inequality.  Multiplying both sides by $2\pi (r-k^{-1})^2r^3e^{-r^2} dr$, using the inequality $r-k^{-1} \ge 1$, and integrating over $r \in [2,\infty)$, we have 
\[
2\pi k \int _2^{\infty} r^3 e^{-r^2} dr \le   \int _{|\zeta - z_{2k}| \ge 2} |F(\zeta, 0)|^2e^{-|\zeta|^2} dA(\zeta) \le \int _{\C} |F_k(\zeta , 0)|^2e^{-|\zeta|^2} dA(\zeta) \le C.
\]
The desired contradiction is obtained as soon as $k$ is large enough.
\red
\end{ex}

It would be interesting to find the right necessary geometric condition on hypersurfaces $W$ for the surjectivity of $\sr _W$.  It seems that the conditions should not be too local in nature, as uniform flatness seems to be.


\begin{thebibliography}{99}

%\bibitem[AGV-1985]{agv} Arnold, V.I.; Guzein-Zade, S.M.; Varchenko, A.N., Singularities of differentiable maps. Vol. I \& II.  Monographs in Mathematics, 82. BirkhŠuser Boston, Inc., Boston, MA, 1985.

\bibitem[BO-1995]{quimbo}  Berndtsson, B.; Ortega Cerd\` a, J., {\it On interpolation and sampling in Hilbert spaces of analytic functions.} J. Reine Angew. Math.  464  (1995), 109--128.

%\bibitem[B-64]{bishop64} Bishop, E., \textit{Conditions for the analyticity of certian sets.} Michigan Math. J. 11 (1964), 289--304.

%\bibitem[DG-1988]{dg} Daubechies, I.;  Grossmann, A., {\it Frames in the Bargmann space of entire functions.}   Comm. Pure Appl. Math., 41 (1988), 151Ð164. 

%\bibitem[DGM-1986]{dgm} Daubechies, I.;  Grossmann, A.;  Meyer, Y.,  {\it Painless nonorthogonal expansions.}  J. Math. Phys. 27 (1986), 1271Ð1283.

%\bibitem[D-2000]{dem-trieste} Demailly, J.-P., {\it Trieste notes}

%\bibitem[DM-1997]{dm1} Diederich, K., Mazzilli, E., {\it Extension and restriction of holomorphic functions.}  Ann. Inst. Fourier (Grenoble) Vol. 47, No. 4 (1997) 1079-1099.

%\bibitem[DM-2000]{dm2} Diederich, K., Mazzilli, E., {\it A remark on the Theorem of Ohsawa-Takegoshi.} Nagoya Math. J., Vol. 158 (2000), 185Ð189.


\bibitem[FV-2007]{fv} Forg\' acs, T., Varolin, D., {\it Sufficient conditions for interpolation and sampling hypersurfaces in the Bergman ball.} Internat. J. Math. 18 (2007), no. 5, 559--584.

%\bibitem[H\"o-1990]{h} H\"ormander, L., {\it An Introduction to Complex Analysis in Several Variables.} North-Holland, 1990.

\bibitem[La-1967]{l-67} Landau, H. J., {\it Necessary density conditions for sampling and interpolation of certain entire functions.}  Acta Math. 117 (1967), 37-52.

\bibitem[Li-2001]{lindholm} Lindholm, N., {\it Sampling in weighted $L^p$ spaces of entire functions in $\C ^n$ and estimates of the Bergman kernel.} J. Funct. Anal. 182 (2001), no. 2, 390--426.

%\bibitem[MV-2007]{mv} McNeal, J.; Varolin, D.,  {\it Analytic inversion of adjunction: $L^2$ extension theorems with gain}. Ann. Inst. Fourier (Grenoble) 57 (2007), no. 3, 703--718.

%\bibitem[O-1994]{o-4} Ohsawa, T., {\it On the extension of $L^2$ holomorphic functions IV--- a new density concept.}  Geometry and Analysis on Complex Manifolds, Festschrift for Prof. S. Kobayashi, World Sci. 1994, pp. 157Ð170.

\bibitem[O-2009]{o-ober} Ohsawa, T., {\it Lecture at Oberwolfach meeting April 12-18, 2009}.

\bibitem[OSV-2006]{osv}  Ortega-Cerd\`a , J.; Schuster, A.; Varolin, D., {\it Interpolation and sampling hypersurfaces for the Bargmann-Fock space in higher dimensions.} Math. Ann. 335 (2006), no. 1, 79--107.

\bibitem[OS-1998]{quimseep} Ortega-Cerd\`a , J.; Seip, K., {\it Beurling-type density theorems for weighted $L^p$ spaces of  entire functions}, J. Anal. Math. 75 ( 1998 ), 247-266.

%\bibitem[Ostrovsky-2008]{stas} Ostrovsky, S. {\it Interpolation problems for sequences of finite density.}  In preparation.

\bibitem[SV-2012]{sv-toeplitz} Schuster, A., Varolin, D., {\it Toeplitz operators and Carleson measures on generalized Bargmann-Fock spaces.} Integral Equations Operator Theory 72 (2012), no. 3, 363 -- 392.

%\bibitem[Seip-92]{s1}  Seip, K., {\it Density theorems for sampling and interpolation in the Bargmann-Fock space. I.}  J. Reine Angew. Math. 429  (1992), 91--106.

%\bibitem[Seip-93]{s2}  Seip, K., {\it Beurling type density theorems in the unit disk.}  Invent. Math.  113  (1993),  no. 1, 21--39.

%\bibitem[SW-92]{sw} Seip, K.; Wallst\' en, R. {\it Density theorems for sampling and interpolation in the Bargmann-Fock space. II.}  J. Reine Angew. Math.  429  (1992), 107--113.

%\bibitem[Siu-1982]{siu-82} Siu, Y.-T.,   {\it Complex-analyticity of harmonic maps, vanishing and Lefschetz theorems.} J. Diff. Geom. 17 (1982), no. 1, 55--138.

%\bibitem[Siu-1976]{siu-stein} Siu, Y.-T., {\it Every Stein subvariety admits a Stein neighborhood}. Invent. Math. 38 (1976/77), no. 1, 89--100.

\bibitem[V-2007]{v-tak} Varolin, D., {\it A Takayama-type Extension Theorem.}  Compos. Math. 144 (2008), no. 2, 522--540. 

\end{thebibliography}
\end{document}